\documentclass{article}
\usepackage[a4paper, top=0.5in, bottom=1in, left=1in, right=1in]{geometry} 
\usepackage{graphicx} 
\usepackage{caption}  
\usepackage{amsmath}
\usepackage{amsthm}
\usepackage{float}
\usepackage{pifont}
\usepackage{centernot}
\usepackage{enumitem}
\usepackage{bigints} 
\usepackage{ragged2e} 
\usepackage{setspace}
\usepackage[english]{babel}
\usepackage[autostyle]{csquotes}
\usepackage{mathptmx} 
\usepackage{cite}
\usepackage[colorlinks=true, linkcolor=blue, citecolor=blue, urlcolor=blue]{hyperref}

\newtheorem{theorem}{Theorem}[section]

\newtheorem{definition}{Definition}[section]
\newtheorem{example}{Counterexample}[section] 
\newtheorem{eg}{Example}[section]

\newtheorem{result}{Result}[section] 

\title{Variance Residual Life Ageing Intensity Function}

\author{
Ashutosh Singh \\
\small Department of Mathematics and Statistics, Indian Institute of Technology Tirupati-$\textsc{\footnotesize 517619}$, India \\
\textit{\footnotesize E-mail address:} \texttt{\small ma23d501@iittp.ac.in}
}

\date{} 

\begin{document}

\maketitle

\noindent\rule{\textwidth}{1pt}

\noindent\textbf{Abstract}

\begin{flushleft}
\begin{justify}
Quantitative measurement of ageing across systems and components is crucial for accurately assessing reliability and predicting failure probabilities. This measurement supports effective maintenance scheduling, performance optimisation, and cost management. Examining the ageing characteristics of a system that operates beyond a specified time $t > 0$ yields valuable insights. This paper introduces a novel metric for ageing, termed the Variance Residual Life Ageing Intensity (VRLAI) function, and explores its properties across various probability distributions. Additionally, we characterise the closure properties of the two ageing classes defined by the VRLAI function. We propose a new ordering, called the Variance Residual Life Ageing Intensity (VRLAI) ordering, and discuss its various properties. Furthermore, we examine the closure of the VRLAI order under coherent systems.

\end{justify}

\end{flushleft}

\noindent\rule{\textwidth}{1pt}

\vspace{10pt}


\section{Introduction }
\begin{justify}
In reliability theory, ageing is described as how the probability of failure of a system or component increases over time. Mathematically, ageing is characterised by an increasing failure rate (IFR), where the hazard function $r(t) = \frac{f(t)}{\bar{F}(t)}$ is seen to rise with time, indicating that the system's risk of failure is understood to grow as it ages. Here, $f(t)$ is represented as the probability density function of failure time, and $\bar{F}(t)$ is represented as the survival function. This concept is used to aid in modelling and predicting the reliability of systems, providing insights into how the failure probability is observed to evolve with the system’s age.\\

Jiang et al.~(2003)~\cite{jiang2003aging} introduced the concept of Ageing Intensity (AI) to quantitatively assess ageing properties, defining AI as the ratio of the failure rate to a baseline failure rate. Nanda et al.~(2007)~\cite{nanda2007properties} discussed the properties of the AI function and analysed its behaviour across various distributions. Bhattacharjee et al.~(2013)~\cite{bhattacharjee2013reliability} emphasised that the AI function provided a quantitative approach to evaluating system ageing, playing a key role in the analysis of generalised Weibull models and system reliability.\\

Magdalena (2018)~\cite{SZYMKOWIAK2018198} introduced a family of generalised ageing intensity functions to characterise the lifetime distributions of univariate positive, absolutely continuous random variables. Buono et al.~(2021)~\cite{buono2021generalized} proposed a family of generalised reversed ageing intensity functions, showing that if the parameter was positive, it uniquely characterised the distribution functions of univariate positive absolutely continuous random variables. Francesco Buono (2022)~\cite{buono2022multivariate} extended the concept of ageing intensity functions to the multivariate case using multivariate conditional hazard rate functions.
\\

The quantitative measures of ageing that we discussed were derived from the hazard rate function, while the residual life function played a key role in analysing ageing behaviour. It was important to study the ageing behaviour of a system that had already operated for a specified duration $t > 0$.\\
\begin{justify}
In the literature, the ageing intensity function has been explored based on the mean residual life (MRL). Ashutosh et al.~(2024) introduced a new ageing intensity function, referred to as the Mean Residual Life Ageing Intensity (MRLAI) function, \( L_X^{\mu}(t) \), defined as

\begin{equation*}
    L_X^{\mu}(t) = \frac{\mu_X(t)}{\frac{1}{t}\int_0^t \mu_X(u) \, du}, \quad \text{for $0 < t < \infty$},
\end{equation*}
where \( \mu_X(t) \) denotes the mean residual life function, and its characteristics have been analysed across various distributions.\\

Earlier, Launer (1984)~\cite{launer1984inequalities} introduced the concept of variance residual life (VRL) and investigated a class of distributions characterised by either increasing or decreasing variance in residual life. As the variance residual life function is based on two moments of residual function, it provides a more precise representation of a system or component's ageing behaviour.
\end{justify}
\begin{justify}
This paper provides a detailed discussion on the ageing intensity function based on the variance of residual life. If $X$ represents the lifetime of a component, then the random variable $X_{t} = X - t \mid X > t$ is referred to as the residual life after time $t$. The mean value of $X_{t}$ is given by

\begin{equation*}
    \mu_{X}(t) = E(X_{t}) = E(X - t \mid X > t) = \dfrac{\int_{t}^{\infty}\bar{F_{X}}(x) \, dx}{\bar{F_{X}}(t)}, \quad t \geq 0.
\end{equation*}

It is assumed that $\mu_{X}(0) = E(X) < \infty$ and $E(X^{2}) < \infty$. The variance of $X_{t}$ is defined as

\begin{equation*}
  \begin{split}
      \sigma^{2}_{X}(t) = \text{Var}(X_{t}) &= \text{Var}(X - t \mid X > t)\\
      &= \frac{2}{\bar{F_{X}}(t)}\int_{t}^{\infty}\int_{x}^{\infty}\bar{F_{X}}(x) \, dy \, dx - \mu_{X}^{2}(t)\\
      &= \frac{2}{\bar{F_{X}}(t)}\int_{t}^{\infty}\mu_{X}(x)\bar{F_{X}}(x) \, dx - \mu_{X}^{2}(t), \quad t \geq 0.
  \end{split} 
\end{equation*}

The variance of $X_{t}$ is referred to as the Variance Residual Life (VRL) function.

\end{justify}
\begin{justify}
 Abouammoh et al. (1990)~\cite{ABOUAMMOH1990751} established a relationship between the mean residual life (MRL) function, $\mu_X(t)$, the variance residual life (VRL) function, $\sigma^2_X(t)$, and the survival function. The result demonstrates that, with $\mu_X(t)$ representing the MRL function and $\sigma^2_X(t)$ representing the VRL function, the survival function $\bar{F_X}(t)$ takes the form:
\begin{result}\label{rs:1.1}
\begin{equation*}
    \bar{F_X}(t) = e^{-\bigintsss_{0}^{t}\dfrac{\frac{d}{dt}\sigma^{2}_{X}(t)}{\sigma^{2}_{X}(t) - \mu_{X}^{2}(t)} \, dx}.
\end{equation*}
\end{result}

\end{justify}
\begin{justify}
Building on this, Cox (1962)~\cite{cox1962renewal} and Meilijson (1972)~\cite{meilijson1972limiting} explored the relationship between the MRL function and the survival function. The analysis establishes that, under the conditions $F(0) = 0$ and $F$ being right continuous, the relationship between the reliability function and the MRL function is given by:
\begin{result}\label{rs:1.2}
\[
\bar{F}(t) = \begin{cases} 
       \frac{\mu_{X}(0)}{\mu_{X}(t)} e^{-\int_{0}^{t} \frac{1}{\mu_{X}(u)} \, du} & \text{for } 0 \leq t < F^{-1}(1), \\
       0 & \text{for } F^{-1}(1) \leq t < \infty,
   \end{cases}
\]
where $F^{-1}(1) = \sup\{t \mid F(t) < 1\}$.
\end{result}

\end{justify}
\begin{justify}
Additionally, Gupta(1987) \cite{gupta1987monotonic} derived the relationship between the hazard rate function, MRL function, and VRL function. It is given by:
\begin{result}\label{rs:1.3}
    \[\frac{d}{dt} \sigma^{2}_{X}(t) = r_{X}(t)\left( \sigma^{2}_{X}(t) - \mu_{X}^{2}(t)\right)\]
\end{result}    
\end{justify}
\begin{justify}
Further, Gupta (2006)~\cite{gupta2006variance} investigated the connection between the monotonic behaviour of the mean residual life classes and the variance residual life classes. 
\begin{result}\label{res:1.4}
 DMRL(IMRL) property implies the DVRL(IVRL) property.
\end{result}    
\end{justify}
\vspace{0.1cm}
This section is concluded by referencing additional sources on quantitative measures of ageing. Notably, a comprehensive overview of these measures is provided by Magdalena Szymkowiak (2020)~\cite{sunoj2020ageing}, while a detailed discussion on stochastic ageing is offered by Lai et al. (2006)~\cite{lai2006stochastic}. Furthermore, a brief discussion on the different stochastic orders available is found in Shaked and Shantikumar (1994)~\cite{shaked2007stochastic}.\\
This manuscript is structured as follows: Section 2 defines the variance residual life function and its key properties are investigated. Section 3 examines the closure properties of two classes—Increasing VRLAI and Decreasing VRLAI. In Section 4, a new order, termed VRLAI order, is introduced and its properties are rigorously analysed. Furthermore, the closure properties of VRLAI in the context of coherent systems are explored in detail.
\end{justify}
\section{Variance Residual Life Ageing Intensity Function}
\begin{definition}
For a non-negative random variable $X$, the Variance Residual Life Ageing Intensity (VRLAI) function is defined as the ratio of the variance of $X_{t}$ to its average value. 
\begin{equation*}
    L_{X}^{\sigma^{2}}(t) = \frac{t \cdot \sigma^{2}_{X}(t)}{\int_{0}^{t} \sigma^{2}_{X}(u) \, du}
\end{equation*} 
\end{definition}
It can be observed that:
\begin{itemize}
    \item If \( L_{X}^{\sigma^{2}}(t) > 1 \), the uncertainty of the residual life at time $t$ is greater than the average uncertainty up to time $t$.
    \item If \( L_{X}^{\sigma^{2}}(t) = 1 \), the uncertainty of the residual life at time $t$ is equal to the average uncertainty up to time $t$.
    \item If \( L_{X}^{\sigma^{2}}(t) < 1 \), the uncertainty of the residual life at time $t$ is less than the average uncertainty up to time $t$.
\end{itemize}

\subsection{Distribution Characterisation through VRLAI Function}
\begin{theorem}\label{thm:2.1}
    For a non negative random variable $X$, $ L_{X}^{\sigma^{2}}(t) = 1$ for $t\geq 0$ if and only if $X$ fallows exponential distribution.
\end{theorem} 
\begin{proof}
Let \( L_X^{\sigma^{2}}(t) = 1 \) if and only if 
\[
L_X^{\sigma^{2}}(t) = \frac{t \cdot \sigma^{2}_{X}(t)}{\int_{0}^{t} \sigma^{2}_{X}(u) \, du} = 1.
\]
\begin{align*}
&\Leftrightarrow  \int_{0}^{t} \sigma^{2}_{X}(u) \, du = t \cdot \sigma^{2}_{X}(t) \quad \forall t \\
&\Leftrightarrow  \frac{d}{dt} \int_{0}^{t} \sigma^{2}_{X}(u) \, du = \frac{d}{dt} (t \cdot \sigma^{2}_{X}(t)) \\
&\Leftrightarrow  \sigma^{2}_{X}(t) = \sigma^{2}_{X}(t) + t \cdot \frac{d}{dt} \left( \sigma^{2}_{X}(t) \right) \\
&\Leftrightarrow  t \cdot \frac{d}{dt} \left( \sigma^{2}_{X}(t) \right) = 0 \\
&\Leftrightarrow  \frac{d}{dt} \left( \sigma^{2}_{X}(t) \right) = 0 \\
&\Leftrightarrow  \sigma^{2}_{X}(t) = \lambda  \quad \forall t 
\end{align*}
where \(\lambda\) is a constant. Now, using result (\ref{rs:1.3}), we get \(\mu_{X}(t) = \lambda\). Furthermore, from result (\ref{rs:1.2}), we obtain
\[
\bar{F}_{X}(t) = e^{-\frac{t}{\lambda}},
\]
which is the survival function of the exponential distribution.\\
Conversely, if \(\sigma^{2}_{X}(t) = \lambda\), then
\[
L_X^{\sigma^{2}}(t) = \frac{t \cdot \lambda}{\int_{0}^{t} \lambda \, du} = \frac{t \cdot \lambda}{t \cdot \lambda} = 1.
\]
\end{proof}

\begin{theorem} \label{thm:2.2}
  If \( X \) follows a Pareto distribution, then \( L_{X}^{\sigma^2}(t) = 3 \) for all \( t \geq 1 \).

\end{theorem}
\begin{proof}
 The density function of $X$ is given by
\[ f_{X}(t) = \begin{cases} 
      \dfrac{(a+1)b^{a+1}}{t^{a+2}} & t\geq b \\
       0 & t < b 
   \end{cases}
\]
The survival function is given by
\[ \bar{F_{X}}(t) = \begin{cases} 
      \left(\dfrac{b}{t}\right)^{a+1} & b \leq t < \infty \\
       1 & -\infty < t < b
   \end{cases}
\]   
Now,
\begin{eqnarray*}
    \mu_{X}(t) &=& \dfrac{\int_{t}^{\infty}\bar{F_{X}}(x) dx}{\bar{F_{X}}(t) }\\
    &=& \dfrac{\int_{t}^{\infty} \left(\dfrac{b}{x}\right)^{a +1}dx}{\left(\dfrac{b}{t}\right)^{a+1} }\\
    &=& \dfrac{t}{a}  \hspace{1cm} a >  1
\end{eqnarray*}
Then, \[\sigma^{2}_{X}(t) = \dfrac{t^{2}(3a-1)}{a^{2}(a-1)}\hspace{1cm} a > 1\]
\[ L_{X}^{\sigma^{2}}(t) = \dfrac{t.\sigma^{2}_{X}(t)}{\int_{0}^{t}\sigma^{2}_{X}(u)du}  = 3.\] 

\end{proof}
\subsection{Properties of VRLAI Function }
\begin{theorem}\label{thm:2.3}
    For a non negative random variable $X$, $ L_{X}^{\sigma^{2}}(t) = 1$ for $t\geq 0$ if and only if  $\sigma_{X}^{2}(t)$ is constant in t.
\end{theorem}
\begin{proof}
This implies that \( X \) follows an exponential distribution, as demonstrated by Theorem~\ref{thm:2.1}.

\end{proof}
\begin{theorem}\label{thm:2.4}
   For a non negative random variable $X$,if $\sigma_{X}^{2}(t)$ is increasing in  $t>0$ then  $L_X^{\sigma^{2}}(t) > 1$  but the converse may  not true.
\end{theorem}
\begin{proof}
  The proof is straightforward and, therefore, omitted. The converse can be demonstrated using the following counterexample.
  \end{proof}
\begin{example}\label{cg:2.1}
 Let $\bar{F_{X}}(t)$   be the survival function defined by 
 \[\bar{F_{X}}(t) = \begin{cases}
     e^{-t}, & 0 \leq t \leq 1\\
     \dfrac{e^{-1}}{t^{2}}, & t > 1
 \end{cases}\]
 Now, it's mean residual life function is 
  \[\mu_{X}(t) = \begin{cases}
    1, & 0 \leq t \leq 1\\
    t, & t>1
\end{cases} \]  
Then, \[L_{\mu}^{X}(t) = \begin{cases}
    1, & 0 \leq t \leq 1\\
  \dfrac{(1+t^{2})}{2}, & t > 1
\end{cases}\]
 which is a monotonically increasing function in $t$.The variance residual life (VRL) function is given by
   \[ \sigma_{X}^{2}(t) = \begin{cases}
        2\mathrm{e}^{-2t}-1,  & 0 \leq t \leq 1\\
        t^2\cdot\left(2\mathrm{e}^{1-t}-1\right), & t>1
    \end{cases}\]
Hence, the VRLAI function of $X$ is 
\[L_{X}^{\sigma^{2}}(t) = \begin{cases}
   \dfrac{t\cdot\left(\mathrm{e}^{2t}-2\right)}{\left(t-1\right)\mathrm{e}^{2t}+1}, & 0 \leq t \leq 1\\
   \dfrac{3\mathrm{e}^2t^3\cdot\left(\mathrm{e}^t-2\mathrm{e}\right)}{\mathrm{e}^2\cdot\left(\left(t^3-31\right)\mathrm{e}^t+6\mathrm{e}\cdot\left(t^2+2t+2\right)\right)+3\mathrm{e}^t}, & t > 1  
\end{cases}\]
It can be seen from Figure~(\ref{fig:1}) that \(\sigma^{2}_{X}(t)\) is not an increasing function of \(t\), but \(L_{X}^{\sigma^{2}}(t) > 1\).
\begin{figure}[H] 
    \centering
    \includegraphics[width=0.7\textwidth]{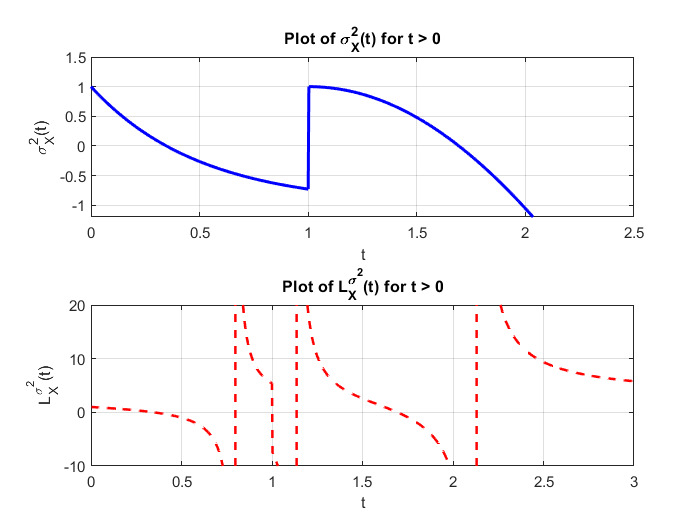}
    \caption{Plot of \(\sigma_{X}^{2}(t)\)  and \(L_{X}^{\sigma^{2}}(t) \)}
    \label{fig:1}
\end{figure}

\end{example}
\begin{theorem}\label{thm:2.5}
For a non negative random variable $X$,if $\sigma_{X}^{2}(t)$ is decreasing in $t>0$  then  $L_X^{\sigma^{2}}(t) < 1$ but the converse may not be true.
\end{theorem}
\begin{proof}
  The proof is straightforward and is therefore omitted. The converse can be demonstrated using the following counterexample.
\end{proof} 
\begin{example}\label{cg:2.2}
    Let $\bar{F_{X}}(t)$   be the survival function defined by 
    \[\bar{F_{X}}(t) = \begin{cases}
        e^{-t}, & 0 \leq t \leq 1\\
        e^{(e-1)+(t-e^{t})} & t > 1
    \end{cases}\] 
Now, 
\[\mu_{X}(t) = \begin{cases}
    1, & 0 \leq t \leq 1\\
    e^{-t}, & t>1
\end{cases} \]
Then, \[L_{\mu}^{X}(t) = \begin{cases}
    1, & 0 \leq t \leq 1\\
    \dfrac{t.e^{-t}}{1 + e^{-1} - e^{-t}}, & t>1
\end{cases}\]
which is a monotonically decreasing function. 
Also,
\[\sigma_{X}^{2}(t) = \begin{cases}
    1, & 0 \leq t \leq 1\\
    2 + (2 - e^{-t})e^{-t}, & t>1
\end{cases}\]
Then, 
\[L_{X}^{\sigma^{2}}(t) =\begin{cases}
    1, & 0 \leq t \leq 1\\
\dfrac{2\mathrm{e}^2t\cdot\left(2\mathrm{e}^t\cdot\left(\mathrm{e}^t+1\right)-1\right)}{\mathrm{e}^t\cdot\left(2\mathrm{e}\cdot\left(\mathrm{e}\cdot\left(\left(2t-1\right)\mathrm{e}^t-2\right)+2\mathrm{e}^t\right)-\mathrm{e}^t\right)+\mathrm{e}^2} , & t > 1  
\end{cases}\]
It can be seen from Figure (\ref{fig:2}) that $\sigma^{2}_{X}(t)$ is not a decreasing function in $t$, but $L_{X}^{\sigma^{2}}(t) < 1$.
 \begin{figure}[H] 
    \centering
    \includegraphics[width=0.7\textwidth]{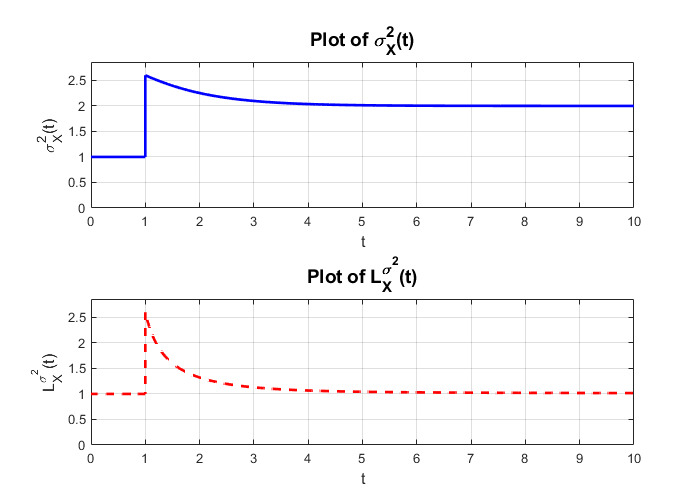}
    \caption{Plot of \(\sigma_{X}^{2}(t)\)  and \(L_{X}^{\sigma^{2}}(t) \)}
    \label{fig:2}
\end{figure}
\end{example}
\begin{theorem}
 If $\sigma^{2}_{X}(t)$  is monotonic function then $L_{X}^{\sigma^{2}}(t)$ need not be monotonic. 
\end{theorem}
\begin{justify}
Three examples have been provided to illustrate this theorem.
Example~(\ref{eg:2.1}) illustrates that both \(\sigma^{2}_{X}(t)\) and \(L_{X}^{\sigma^{2}}(t)\) can exhibit monotonic behaviour. However, counterexamples~(\ref{cg:2.3}) and~(\ref{cg:2.4}) demonstrate that this is not always the case. 
\end{justify}
 \begin{eg}\label{eg:2.1}
  Let \(\bar{F_{X}}(t)\) be the survival function defined by 
  \[
  \bar{F_{X}}(t) = \frac{1}{(t+1)^3}, \quad t \geq 0
  \]
  and 
  \[
  \mu_{X}(t) = \frac{t+1}{2}
  \]
  \[
  \sigma_{X}^{2}(t) = \frac{3(t+1)^2}{4}
  \]
  which is a monotonically increasing function for \(t \geq 0\). Now,
  \[
  L_{X}^{\sigma^{2}}(t) = \frac{3(t+1)^2}{t^2 + 3t + 3}
  \]
  It can be seen from Figure (\ref{fig:3}) that \(L_{X}^{\sigma^{2}}(t)\) is also a monotonically increasing function for \(t \geq 0\).
  \begin{figure}[H] 
    \centering
\includegraphics[width=0.6\textwidth]{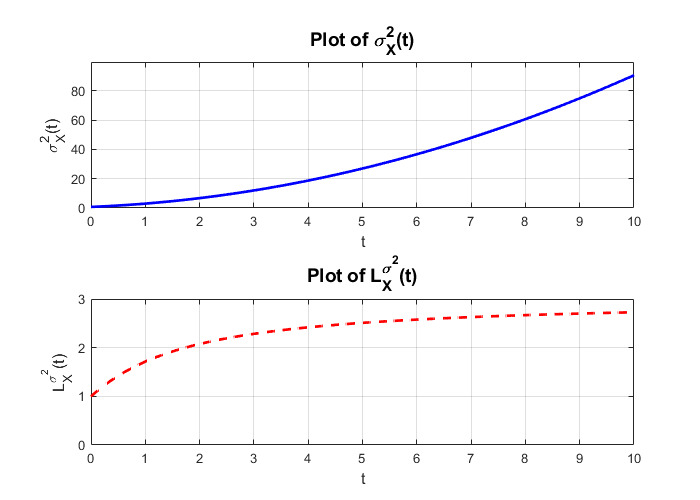}
    \caption{Plot of \(\sigma_{X}^{2}(t)\)  and \(L_{X}^{\sigma^{2}}(t) \)}
    \label{fig:3}
\end{figure}
\end{eg} 
\begin{example}\label{cg:2.3}
  Let \(\bar{F_{X}}(t)\) be the survival function defined by 
  \[
  \bar{F_{X}}(t) = \begin{cases}
      1, & 0 \leq t \leq \frac{1}{2} \\
      \exp{\left(-\left(\ln{2} - \frac{1}{2} + t\right)\right)}, & t > \frac{1}{2}
  \end{cases}
  \]
  Now, 
  \[
  \mu_{X}(t) = \begin{cases}
      1 - t, & 0 \leq t \leq \frac{1}{2} \\
      1, & t > \frac{1}{2}
  \end{cases}
  \]
  and 
  \[
  \sigma_{X}^{2}(t) = \begin{cases}
      \frac{3}{4}, & 0 \leq t \leq \frac{1}{2} \\
      1, & t > \frac{1}{2}
  \end{cases}
  \]
  which is a monotonic (increasing) function. However, 
  \[
  L_{X}^{\sigma^{2}}(t) = \begin{cases}
      1, & 0 \leq t \leq \frac{1}{2} \\
      \frac{1}{1 - \frac{1}{8t}}, & t > \frac{1}{2}
  \end{cases}
  \]
  it can be easily verified that \(L_{X}^{\sigma^{2}}(t)\) is a non-monotonic function.
  \begin{figure}[H] 
    \centering
    \includegraphics[width=0.7\textwidth]{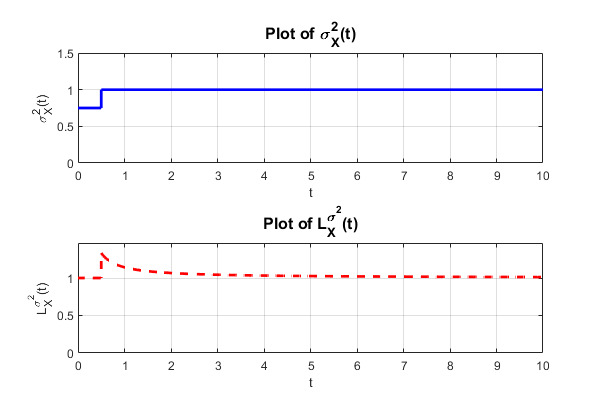}
    \caption{Plot of \(\sigma_{X}^{2}(t)\)  and \(L_{X}^{\sigma^{2}}(t) \)}
    \label{fig:4}
\end{figure}
\end{example}
  \begin{example}\label{cg:2.4}
     Let $\bar{F_{X}}(t)$   be the survival function defined by 
     \[\bar{F_{X}}(t) = \begin{cases}
         e^{-2t} & 0 \leq t < \frac{1}{3}\\
         e^{-\frac{2}{3}} & \frac{1}{3}\leq t < \frac{2}{3}\\
         e^{-2t + \frac{2}{3}} & t \geq \frac{2}{3}
     \end{cases}\]
    Now,
  \[\mu_{X}(t) = \begin{cases}
      \frac{1}{2} + \frac{1}{3}e^{2t - \frac{2}{3}}, & 0 \leq t \leq \frac{1}{3}\\
      \frac{7}{6} - t, &  \frac{1}{3}\leq t < \frac{2}{3}\\
      \frac{1}{2}, & t \geq \frac{2}{3}
  \end{cases} \]
  and  
  \[\sigma_{X}^{2}(t) = \begin{cases}
      \frac{1}{4}+ \left(\frac{1-2t}{3}\right)e^{2t - \frac{2}{3}} - (\frac{1}{9})e^{4t - \frac{4}{3}}, & 0 \leq t < \frac{1}{3}\\
       \frac{1}{4}, & t \geq  \frac{1}{3}
  \end{cases}\]
 which is a monotonic(decreasing) function.  But
 \[L_{X}^{\sigma^{2}}(t) = \begin{cases}
\dfrac{t\cdot\left(4\mathrm{e}^{2t}\cdot\left(\mathrm{e}^{2t}+3\mathrm{e}^\frac{2}{3}\cdot\left(2t-1\right)\right)-9\mathrm{e}^\frac{4}{3}\right)}{\mathrm{e}^{4t}+12\mathrm{e}^\frac{2}{3}\cdot\left(\left(t-1\right)\mathrm{e}^{2t}+1\right)-9\mathrm{e}^\frac{4}{3}t-1}, & 0 \leq t < \frac{1}{3}\\
    \dfrac{9\mathrm{e}^\frac{4}{3}t}{\mathrm{e}^\frac{2}{3}\cdot\left(\mathrm{e}^\frac{2}{3}\cdot\left(9t+7\right)-12\right)+1}, & t \geq \frac{1}{3}
 \end{cases}\] 
 which can be seen from Figure (\ref{fig:5}) that $L_{X}^{\sigma^{2}}(t)$ is a non monotone function.
 \begin{figure}[H] 
    \centering
    \includegraphics[width=0.7\textwidth]{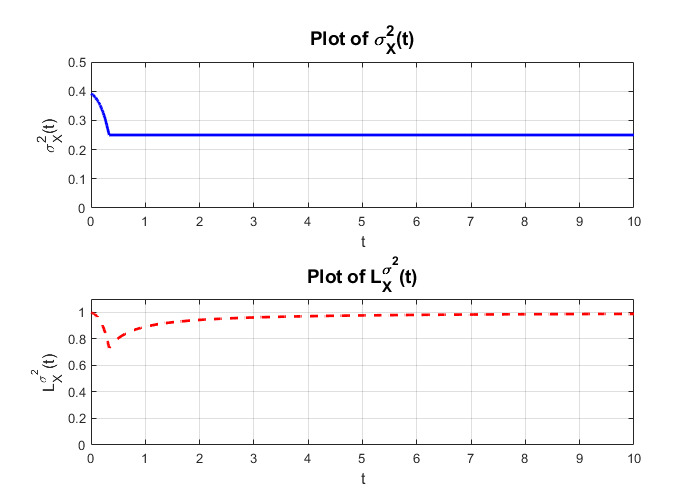}
    \caption{Plot of \(\sigma_{X}^{2}(t)\)  and \(L_{X}^{\sigma^{2}}(t) \)}
    \label{fig:5}
\end{figure}
\end{example}  

\begin{theorem}
  $\sigma_{X}^{2}(t)$ is monotonic $\centernot\implies \mu_{X}(t)$  is monotonic.   
\end{theorem}
\begin{proof}
   Using counterexample (\ref{cg:2.3}), it can be verified that \( \sigma_{X}^{2}(t) \) is monotonic, but \( \mu_{X}(t) \) is non-monotonic.

\end{proof}
\begin{theorem}
    If $L_{\mu}^{X}(t) > 1  \implies L_{X}^{\sigma^{2}}(t) > 1 $  .
\end{theorem}
\begin{proof}
  Let  $L_{\mu}^{X}(t) > 1 $ then  $\mu_{X}(t)$ is increasing in t. Using the result (\ref{res:1.4}), $\sigma_{X}^{2}(t) $ is increasing in t, again by using theorem (\ref{thm:2.4}), $L_{X}^{\sigma^{2}}(t) > 1$.
\end{proof}
\begin{theorem}
    If $L_{\mu}^{X}(t) < 1  \implies L_{X}^{\sigma^{2}}(t) < 1 $  .
\end{theorem}
\begin{proof}
  Let  $L_{\mu}^{X}(t) < 1 $,  then  $\mu_{X}(t)$ is decreasing in t . Using the result (\ref{res:1.4}), $\sigma_{X}^{2}(t) $ is decreasing in t, by using theorem (\ref{thm:2.5}), $L_{X}^{\sigma^{2}}(t) < 1$.
\end{proof} 
\begin{definition}
    A random variable $X$ is said to be increasing in variance residual life ageing intensity(IVRLAI) if the corresponding VRLAI function $L_{X}^{\sigma^{2}}(t)$  is increasing in $t>0$. We call the random variable $X$ as decreasing in variance residual life ageing intensity (DVRLAI) if $L_{X}^{\sigma^{2}}(t)$ is decreasing in $t>0.$
\end{definition}
\begin{theorem}\label{thm:2.9}
  IMRLAI (DMRLAI) property does not imply the IVRLAI (DVRLAI) property.   
\end{theorem} 
\begin{proof}
From the counterexample (\ref{cg:2.1}), it can be seen that $L_{X}^{\mu}(t)$ is an increasing function in $t$, but $L_{X}^{\sigma^{2}}(t)$ is a nonmonotonic function. Similarly, from counterexample (\ref{cg:2.2}), it can be seen that $L_{\mu}^{X}(t)$ is a decreasing function in $t$, but $L_{X}^{\sigma^{2}}(t)$ is a nonmonotonic function. Hence,
\[
\text{IMRLAI (DMRLAI)} \centernot\Rightarrow \text{IVRLAI (DVRLAI)}
\]
\end{proof}
 \section{Closure Properties of IVRLAI and DVRLAI Classes}
In this section, we have studied the closure properties of the IVRLAI and DVRLAI classes.
 
\begin{theorem}
    Let \( X_{1} \) and \( X_{2} \) be two identical and independently distributed random variables. Then the IVRLAI (DVRLAI) classes are closed under the formation of  a mixture of distributions.
\end{theorem}

\begin{proof}
    Let \( X_{1} \) and \( X_{2} \) be two identical and independently distributed random variables with survival functions \(\bar{F}_{X_{1}}(t)\) and \(\bar{F}_{X_{2}}(t)\), respectively. Suppose \( X_{1} \) and \( X_{2} \) are increasing (decreasing) in VRLAI. Consider a mixture of distributions given by
    \begin{align*}
    \bar{F}_{X}(t) &= \alpha \bar{F}_{X_{1}}(t) + (1-\alpha) \bar{F}_{X_{2}}(t) \\
    \bar{F}_{X}(t) &= \alpha \bar{F}_{X_{1}}(t) + (1-\alpha) \bar{F}_{X_{1}}(t), \quad \text{as} \quad \bar{F}_{X_{1}}(t) = \bar{F}_{X_{2}}(t) \\
    \bar{F}_{X}(t) &= \bar{F}_{X_{1}}(t)
\end{align*}
As $\bar{F}_{X_{1}}(t)$  is increasing (decreasing) in VRLAI. Hence, the IVRLAI (DVRLAI) classes are closed with respect to mixtures.
\end{proof}
Further, it is demonstrated that the closure property under convolution does not hold for the IVRLAI and DVRLAI classes.

\begin{example}
    Let \( X_{1} \) and \( X_{2} \) be two identical and independently distributed random variables, each following an exponential distribution with parameter \( 1 \). The variance residual life ageing intensity (VRLAI) function for these random variables is given by
    \[
    L_{X_{1}}^{\sigma^{2}}(t) = L_{X_{2}}^{\sigma^{2}}(t) = 1,
    \]
    which is clearly a monotonic function. Consider the convolution,  \( X_{c} = X_{1} + X_{2} \). The survival function of \( X_{c} \) is
    \[
    \bar{F}_{X_{c}}(t) = \left(t + 1\right) \mathrm{e}^{-t},
    \]
    with the mean residual life function
    \[
    \mu_{X_{c}}(t) = \frac{t + 2}{t + 1},
    \]
    and the variance residual life function
    \[
    \sigma_{X_{c}}^{2}(t) = \frac{t^{2} + 4t + 2}{\left(t + 1\right)^{2}}.
    \]
    The VRLAI function for \( X_{c} \) is given by
    \[
    L_{X_{c}}^{\sigma^{2}}(t) = \frac{t \left(t^{2} + 4t + 2\right)}{\left(t + 1\right) \left(2 \left(t + 1\right) \ln\left(t + 1\right) + t^{2}\right)}.
    \]
    This function is non-monotone for \( t > 0 \), as illustrated by the following values:
    \[
    L_{X_{c}}^{\sigma^{2}}(4.0) = 0.8475004,
    \]
    \[
    L_{X_{c}}^{\sigma^{2}}(6.0) = 0.8402997,
    \]
    \[
    L_{X_{c}}^{\sigma^{2}}(10.0) = 0.8450919.
    \] 
Thus, the IVRLAI class is not closed under convolution.
\end{example}
\begin{justify}
Further, counterexample (\ref{cg:3.2}) shows that the IVRLAI class is not closed under the formation of a coherent system. 
\end{justify}
\begin{example}\label{cg:3.2}
  Let \( X \) be a random variable following an exponential distribution. For this case, the variance residual life function is \( L_{X}^{\sigma^{2}}(t) = 1 \), which is monotonic in \( t > 0 \). Consider \( f_{X_{(2:3)}}(t) \) as the density function of the $2^{nd}$ order statistic in a sample of size 3 from the distribution of \( X \). The density function is given by
  \[
  f_{X_{(2:3)}}(t) = 6 \mathrm{e}^{-3t} \left(\mathrm{e}^{t} - 1\right) \quad \text{for} \quad t \geq 0
  \]
  and its survival function is
  \[
  \bar{F}_{X_{(2:3)}}(t) = \mathrm{e}^{-3t} \left(3 \mathrm{e}^{t} - 2\right).
  \]
  The mean residual life function of the  $2^{nd}$ order statistic is
  \[
  \mu_{X_{(2:3)}}(t) = \frac{9 \mathrm{e}^{t} - 4}{6 \left(3 \mathrm{e}^{t} - 2\right)},
  \]
  and the variance residual life function is
  \[
  \sigma^{2}_{X_{(2:3)}}(t) = \frac{\left(15 \mathrm{e}^{t} - 10\right) \ln\left(3 \mathrm{e}^{t} - 2\right) + \left(12t - 6\right) \mathrm{e}^{t} - 8t + 6}{108 \mathrm{e}^{t} - 72}.
  \]
  The variance residual life aging intensity function of \( X_{(2:3)} \) is given by
  \[
  L_{X_{(2:3)}}^{\sigma^{2}}(t) = \frac{t \left(81 \mathrm{e}^{2t} - 84 \mathrm{e}^{t} + 16\right)}{\left(3 \mathrm{e}^{t} - 2\right) \left(\left(15 \mathrm{e}^{t} - 10\right) \ln\left(3 \mathrm{e}^{t} - 2\right) + \left(12t - 6\right) \mathrm{e}^{t} - 8t + 6\right)}.
  \]
  This function is non-monotone, as evidenced by the following values:
  \[
  L_{X_{(2:3)}}^{\sigma^{2}}(0.5) = 0.92533,
  \]
  \[
  L_{X_{(2:3)}}^{\sigma^{2}}(1.5) = 0.88633,
  \]
  \[
  L_{X_{(2:3)}}^{\sigma^{2}}(3.5) = 0.91047.
  \] 
  Thus, the IVRLAI class is not closed under the formation of a coherent system.
\end{example}

 \section{VRLAI Order }
   We define a probabilistic order based on the variance residual life ageing intensity function \( L_{X}^{\sigma^{2}}(t) \) as follows.

 \begin{definition}
     A random variable $X$ is said to be smaller than random variable $Y$ in the VRLAI order (denoted by $X \underset{VRLAI}{\preceq} Y $) if $ L_{X}^{\sigma^{2}}(t) \preceq  L_{Y}^{\sigma^{2}}(t)$, for all $t>0$.
 \end{definition}  
A class of parametric distributions is identified for which the VRLAI order is present between the random variables.
 \begin{eg}
     Let \( X_{i} \) be a random variable having a Pareto distribution with a survival function
\[
\bar{F}_{X_{i}}(t) = \left(\dfrac{1}{t}\right)^{\alpha_{i}}, \quad \alpha_{i} > 0, \quad t > 0, \quad i = 1, 2.
\]
If \( \alpha_{1} \leq \alpha_{2} \), then
\[
X_{1} \underset{VRLAI}{\preceq} X_{2}.
\]

 \end{eg} 
 \subsection{Properties of VRLAI Order}
\begin{theorem}\label{thm:equ}
    For two random variables \(X\) and \(Y\), the following conditions are equivalent:
    \begin{enumerate}[label=(\roman*)]
        \item \(X \underset{VRLAI}{\preceq} Y\),
        
        \item The ratio \(\frac{\int_{0}^{t} \sigma_{X}^{2}(u)\,du}{\int_{0}^{t} \sigma_{Y}^{2}(u)\,du}\) is decreasing for all \(t > 0\).
    \end{enumerate}
\end{theorem}

\begin{proof}
   (i) implies (ii): \\
    Assume \( X \underset{VRLAI}{\preceq} Y \), which implies 
    \[
    \frac{t \cdot \sigma_{X}^{2}(t)}{\int_{0}^{t} \sigma_{X}^{2}(u) \, du} \leq \frac{t \cdot \sigma_{Y}^{2}(t)}{\int_{0}^{t} \sigma_{Y}^{2}(u) \, du}.
    \]
    Let \( H_{X}(t) = \int_{0}^{t} \sigma_{X}^{2}(u) \, du \) and \( H_{Y}(t) = \int_{0}^{t} \sigma_{Y}^{2}(u) \, du \). Then  \( H_{X}'(t) = \sigma^{2}_{X}(t) \) and \( H_{Y}'(t) = \sigma^{2}_{Y}(t) \), it follows that:
    \[
    H_{X}(t) \cdot H_{Y}'(t) - H_{Y}(t) \cdot H_{X}'(t) \geq 0.
    \]
    This inequality can be rewritten as:
    \[
    \frac{H_{X}(t) \cdot H_{Y}'(t) - H_{Y}(t) \cdot H_{X}'(t)}{(H_{X}(t))^2} \geq 0,
    \]
    which implies
    \[
    g'(t) \geq 0, \quad \text{where} \quad g(t) = \frac{H_{Y}(t)}{H_{X}(t)}.
    \]
    Thus, \( g(t) = \frac{H_{Y}(t)}{H_{X}(t)} \) is an increasing function for \( t > 0 \), which implies that $ g(t) = \frac{H_{X}(t)}{H_{Y}(t)}$is a decreasing function for \( t > 0 \). Hence, we conclude that $ \frac{\int_{0}^{t} \sigma_{X}^{2}(u) \, du}{\int_{0}^{t} \sigma_{Y}^{2}(u) \, du}$
    is decreasing in  \( t > 0 \).
\end{proof}
The reflexive, commutative, and antisymmetric properties of the VRLAI order can be explained as follows:

\begin{theorem}
    The following properties hold for variance residual life average increasing (VRLAI) ordering:
    \begin{enumerate}[label=(\roman*)]
        \item \(X \underset{VRLAI}{\preceq} X\): For any random variable \(X\), the VRLAI ordering holds trivially, meaning \(X\) is VRLAI ordered with respect to itself.
        
        \item If \(X \underset{VRLAI}{\preceq} Y\) and \(Y \underset{VRLAI}{\preceq} Z\), then \(X \underset{VRLAI}{\preceq} Z\): That is, VRLAI ordering is transitive. If \(X\) is VRLAI ordered with respect to \(Y\), and \(Y\) is VRLAI ordered with respect to \(Z\), then \(X\) is VRLAI ordered with respect to \(Z\).

        \item If \(X \underset{VRLAI}{\preceq} Y\) and \(Y \underset{VRLAI}{\preceq} X\), then \(X\) and \(Y\) must have proportional variance residual lives. This implies that the VRLAI ordering in both directions indicates that the variance residual lives of \(X\) and \(Y\) are proportional to one another.
    \end{enumerate}
\end{theorem}

\begin{proof}
    \begin{enumerate}[label=(\roman*)]
        \item It is evident that for all $t > 0$, 
        \[
        L_{X}^{\sigma^{2}}(t) \preceq  L_{X}^{\sigma^{2}}(t) \implies X \underset{VRLAI}{\preceq} X
        \]
        demonstrating that the VRLAI ordering holds trivially for any random variable compared with itself.
        
        \item If \(X \underset{VRLAI}{\preceq} Y\), then it follows that
        \[
        L_{X}^{\sigma^{2}}(t) \preceq  L_{Y}^{\sigma^{2}}(t).
        \]
        Similarly, if \(Y \underset{VRLAI}{\preceq} Z\), then
        \[
        L_{Y}^{\sigma^{2}}(t) \preceq  L_{Z}^{\sigma^{2}}(t).
        \]
        Thus,
        \[
        L_{X}^{\sigma^{2}}(t) \preceq  L_{Z}^{\sigma^{2}}(t),
        \]
        which implies that \(X \underset{VRLAI}{\preceq} Z\).

        \item If \(X \underset{VRLAI}{\preceq} Y\), then 
        \[
        L_{X}^{\sigma^{2}}(t) \preceq  L_{Y}^{\sigma^{2}}(t),
        \]
        and if \(Y \underset{VRLAI}{\preceq} X\), then
        \[
        L_{Y}^{\sigma^{2}}(t) \preceq  L_{X}^{\sigma^{2}}(t).
        \]
        Therefore, it must hold that
        \[
        L_{X}^{\sigma^{2}}(t) =  L_{Y}^{\sigma^{2}}(t),
        \]
        implying
        \[
        \frac{t \sigma^{2}_{X}(t)}{\int_{0}^{t} \sigma^{2}_{X}(u) du} = \frac{t \sigma^{2}_{Y}(t)}{\int_{0}^{t} \sigma^{2}_{Y}(u) du}.
        \]
        Let \(H_{X}(t) = \int_{0}^{t} \sigma_{X}^{2}(u) du\) and \(H_{Y}(t) = \int_{0}^{t} \sigma_{Y}^{2}(u) du\). Now, it follows that
        \[
        H_{X}(t) \sigma_{Y}^{2}(t) - H_{Y}(t) \sigma_{X}^{2}(t) = 0.
        \]
        Since \(H_{X}'(t) = \sigma_{X}^{2}(t)\) and \(H_{Y}'(t) = \sigma_{Y}^{2}(t)\), we have
        \[
        H_{X}(t) H_{Y}'(t) - H_{Y}(t) H_{X}'(t) = 0.
        \]
        Dividing both sides by \((H_{Y}(t))^2\), we get
        \[
        \frac{H_{X}(t) H_{Y}'(t) - H_{Y}(t) H_{X}'(t)}{(H_{Y}(t))^2} = 0.
        \]
        Hence, 
        \[
        g'(t) = 0, \quad  \text{where} ~ g(t) = \frac{H_{X}(t)}{H_{Y}(t)},
        \]
        which implies that \(g(t)\) is constant, i.e.,
        \[
        g(t) = c.
        \]
        Therefore, 
        \[
        H_{X}(t) = c H_{Y}(t),
        \]
        which gives
        \[
        \int_{0}^{t} \sigma_{X}^{2}(u) du = c \int_{0}^{t} \sigma^{2}_{Y}(u) du.
        \]
        This leads to
        \[
        \sigma_{X}^{2}(t) = c \sigma_{Y}^{2}(t).
        \]
        Thus, \(X\) and \(Y\) have proportional variance residual lives.
    \end{enumerate}
\end{proof}

The following counterexample illustrates that VRLAI order does not necessarily imply increasing convex (icx) order.
\begin{example}
    Consider the random variable \( X \) that follows an exponential distribution with the survival function
    \[
    \bar{F}_{X}(t) = e^{-\frac{t}{2}}, \quad t \geq 0,
    \]
    and the random variable \( Y \) that follows a Pareto distribution with the survival function
    \[
    \bar{F}_{Y}(t) = \frac{1}{t^{3}}, \quad t \geq 1.
    \]
    It follows that
    \[
    L_{X}^{\sigma^{2}}(t) = 1
    \]
    and
    \[
    L_{Y}^{\sigma^{2}}(t) = 3,
    \]
    for all \( t > 0 \). Since \( L_{X}^{\sigma^{2}}(t) \preceq L_{Y}^{\sigma^{2}}(t) \), we conclude that \( X \underset{VRLAI}{\preceq} Y \).

    Now, define the integrals
    \[
    h(t) = \int_{t}^{\infty} \bar{F}_{X}(u) \, du = 2e^{-\frac{t}{2}},
    \]
    and
    \[
    k(t) = \int_{t}^{\infty} \bar{F}_{Y}(u) \, du = \frac{1}{2t^2}.
    \]
    For \( X \underset{icx}{\preceq} Y \) to hold, we must have \( h(t) \preceq k(t) \) for all \( t \geq 0 \).

    Evaluating at specific values of \( t \), we get:
    \[
    h(0.5) = 1.557602, \quad h(1.0) = 1.213061,
    \]
    \[
    k(0.5) = 2.0, \quad k(1.0) = 0.5.
    \]
    Clearly, \( h(t) \centernot \preceq k(t) \) for all \( t \geq 0 \). Therefore,
    \[
    X \underset{VRLAI}{\preceq} Y \centernot \implies X \underset{icx}{\preceq} Y.
    \]
\end{example}

\begin{example}
    Consider a random variable \( X \) that follows an exponential distribution with survival function
    \[
    \bar{F}_{X}(t) = e^{-3t}, \quad t \geq 0,
    \]
    and a random variable \( Y \) that follows a Pareto distribution with survival function
    \[
    \bar{F}_{Y}(t) = \frac{1}{t^3}, \quad t \geq 1.
    \]
    It can be observed that
    \[
    L_{X}^{\sigma^{2}}(t) = 1
    \]
    and
    \[
    L_{Y}^{\sigma^{2}}(t) = 3,
    \]
    for all \( t > 0 \). Since \( L_{X}^{\sigma^{2}}(t) \preceq L_{Y}^{\sigma^{2}}(t) \), it follows that \( X \underset{VRLAI}{\preceq} Y \).

    Define
    \[
    z(t) = \frac{\int_{t}^{\infty} \int_{u}^{\infty} \bar{F}_{X}(v) \, dv \, du}{\int_{t}^{\infty} \int_{v}^{\infty} \bar{F}_{Y}(v) \, dv \, du} = \frac{2t e^{-3t}}{9}.
    \]
    If \( X \underset{VRL}{\preceq} Y \), then \( z(t) \) should be a decreasing function of \( t \). 

    Evaluating \( z(t) \) at different values of \( t \), we find:
    \[
    z(0.10) = 0.016462627,
    \]
    \[
    z(0.50) = 0.024792240,
    \]
    \[
    z(2.0) = 0.001101668.
    \]
    Clearly, \( z(t) \) is not a decreasing function of \( t \). Therefore,
    \[
    X \underset{VRLAI}{\preceq} Y \centernot \implies X \underset{VRL}{\preceq} Y.
    \]
\end{example}

We have established that variance residual life ageing intensity (VRLAI) ordering does not imply variance residual life (VRL) ordering. Consequently, it follows that VRLAI ordering does not imply mean residual life (MRL) ordering either. The following counterexample illustrates that likelihood ratio ordering does not imply VRLAI ordering.

\begin{example}
    Let \( X \) and \( Y \) be two independent random variables with density functions
    \[
    f_{X}(t) = 25t e^{-5t}, \quad t \geq 0
    \]
    and
    \[
    f_{Y}(t) = 16t e^{-4t}, \quad t \geq 0,
    \]
    respectively. It can be easily verified that the ratio \( \frac{f_{X}(t)}{f_{Y}(t)} \) decreases with \( t \geq 0 \).

    The variance residual life function for \( X \) is given by
    \[
    \sigma_{X}^{2}(t) = \frac{25t^2 + 20t + 2}{25 \left(5t + 1\right)^2}.
    \]
    Hence,
    \[
    a(t) = \int_{0}^{t} \sigma_{X}^{2}(u) \, du = \frac{\left(10t + 2\right) \ln\left(5t + 1\right) + 25t^2}{625t + 125}.
    \]
    Similarly, the variance residual life function for \( Y \) is
    \[
    \sigma_{Y}^{2}(t) = \frac{8t^2 + 8t + 1}{8 \left(4t + 1\right)^2}.
    \]
    Therefore,
    \[
    b(t) = \int_{0}^{t} \sigma_{Y}^{2}(u) \, du = \frac{\left(4t + 1\right) \ln\left(4t + 1\right) + 8t^2}{128t + 32}.
    \]

    Define
    \[
    c(t) = \frac{a(t)}{b(t)} = \frac{\left(128t + 32\right) \left(\left(10t + 2\right) \ln\left(5t + 1\right) + 25t^2\right)}{\left(625t + 125\right) \left(\left(4t + 1\right) \ln\left(4t + 1\right) + 8t^2\right)}.
    \]
    It can be observed that \( c(t) \) is non-monotone as
    \begin{equation*}
        c(0.1) = 0.6358110, \quad c(1.8) = 0.6177504, \quad c(6.5) = 0.6240882.
    \end{equation*}
    Using Theorem \ref{thm:equ}, we have \( X \centernot \leq_{VRLAI} Y \). Thus,
    \[
    X \underset{lr}{\preceq} Y \centernot \implies X \underset{VRLAI}{\preceq} Y.
    \]
\end{example}
\begin{justify}
 Given that the likelihood ratio ordering does not necessarily imply variance residual life ordering (VRLAI), it follows that no other standard orderings---such as failure rate ordering, starting failure rate ordering, or stochastic ordering---imply VRLAI ordering. For a comprehensive discussion on these orderings, see \cite{shaked2007stochastic}.
Now, we provide several conditions under which two random variables satisfy VRLAI ordering.
   
\end{justify}

\begin{theorem}
    Let  $X$ and $Y$ be two non-negative random variables; if $X$ is decreasing in variance residual life average (DVRLA) and $Y$ is increasing in variance residual life average ( IVRLA), then $X \underset{VRLAI}{\preceq}  Y $.
\end{theorem}
\begin{proof}
    Let \( H_{X}(t) = \int_{0}^{t} \sigma_{X}^{2}(u) \, du \). Then \( H_{X}'(t) = \sigma_{X}^{2}(t) \). 

    Define \( g(t) = \frac{H_{Y}(t)}{H_{X}(t)} \). Since \( X \) is decreasing in variance residual life average (DVRLA) and \( Y \) is increasing in variance residual life average (IVRLA), it follows that \( g'(t) \geq 0 \). 

    Now, compute \( g'(t) \):
    \[
    g'(t) = \frac{H_{X}(t) \cdot H_{Y}'(t) - H_{Y}(t) \cdot H_{X}'(t)}{(H_{X}(t))^2}
    \]
    Since \( g'(t) \geq 0 \), we have:
    \[
    H_{X}(t) \cdot H_{Y}'(t) - H_{Y}(t) \cdot H_{X}'(t) \geq 0
    \]
    Substituting \( H_{X}'(t) = \sigma_{X}^{2}(t) \) and \( H_{Y}'(t) = \sigma_{Y}^{2}(t) \), this simplifies to:
    \[
    H_{X}(t) \cdot \sigma_{Y}^{2}(t) - H_{Y}(t) \cdot \sigma_{X}^{2}(t) \geq 0
    \]
    Rearranging gives:
    \[
    \frac{\sigma_{Y}^{2}(t)}{H_{Y}(t)} \geq \frac{\sigma_{X}^{2}(t)}{H_{X}(t)}
    \]
    Which implies:
    \[
    \frac{t \cdot \sigma_{Y}^{2}(t)}{\int_{0}^{t} \sigma_{Y}^{2}(u) \, du} \geq \frac{t \cdot \sigma_{X}^{2}(t)}{\int_{0}^{t} \sigma_{X}^{2}(u) \, du}
    \]
    Therefore:
    \[
    L_{X}^{\sigma^{2}}(t) \preceq L_{Y}^{\sigma^{2}}(t)
    \]
    Hence:
    \[
    X \underset{VRLAI}{\preceq} Y
    \]
\end{proof}
\begin{justify}
 Furthermore, we present a counterexample where \( X \) is not decreasing in variance residual life average (DVRLA), yet \( X \underset{VRLAI}{\preceq} Y \). This counterexample demonstrates that the conditions of \( X \) being DVRLA and \( Y \) being increasing in variance residual life average (IVRLA) are sufficient for \( X \underset{VRLAI}{\preceq} Y \) to hold.   
\end{justify}
\begin{example}
    Let \( X \) and \( Y \) be random variables with survival functions
    \[
    \bar{F}_{X}(t) = \frac{1}{\left(t+1\right)^3}, \quad t \geq 0
    \]
    and
    \[
    \bar{F}_{Y}(t) = \frac{1}{t^{3}}, \quad t \geq 1,
    \]
    respectively. The variance residual life functions are given by
    \[
    \sigma^{2}_{X}(t) = \frac{3\left(t+1\right)^2}{4}, \quad t \geq 0
    \]
    and
    \[
    \sigma^{2}_{Y}(t) = \frac{3t^2}{4}, \quad t \geq 1.
    \]

    Now,
    \[
    \frac{1}{t}\int_{0}^{t} \sigma^{2}_{X}(u) \, du = \frac{t^2 + 3t + 3}{4},
    \]
    and
    \[
    \frac{1}{t}\int_{1}^{t} \sigma^{2}_{Y}(u) \, du = \frac{t^{3} - 1}{4t}.
    \]

    Also,
    \[
    L_{X}^{\sigma^{2}}(t) = \frac{3\left(t+1\right)^2}{t^2 + 3t + 3}.
    \]
    It can be easily verified that \( L_{X}^{\sigma^{2}}(t) < 3 \) and \( L_{Y}^{\sigma^{2}}(t) = 3 \) for all \( t > 0 \). Although both \( X \) and \( Y \) are IVRLA, we have \( X \underset{VRLAI}{\preceq} Y \).
\end{example}

\begin{theorem}
    If \( X \) is decreasing in VRL and \( Y \) is increasing in VRL, then 
    \[
    X \underset{VRLAI}{\preceq} Y.
    \]
\end{theorem} 

\begin{proof}
    If \( X \) is decreasing in VRL, then by theorem (\ref{thm:2.5}), \( L_{X}^{\sigma^{2}}(t) < 1 \). Additionally, if \( Y \) is increasing in VRL, then by theorems (\ref{thm:2.3}) and (\ref{thm:2.4}), \( L_{Y}^{\sigma^{2}}(t) \geq 1 \). Therefore, for all \( t > 0 \),
    \[
    L_{X}^{\sigma^{2}}(t) \leq L_{Y}^{\sigma^{2}}(t) \implies X \underset{VRLAI}{\preceq} Y.
    \]
\end{proof}
Further, we examined the closure of the VRLAI order under the formation of a coherent system.
\begin{example}
\begin{justify}
 The VRLAI ordering is not closed under the formation of a parallel system. To illustrate this, consider the following:

Let \( X_{1} \) and \( X_{2} \) be two independent and identically distributed (i.i.d.) random variables with an Erlang distribution and survival function
\[
\bar{F}_{X}(t) = (1 + t)e^{-t}, \quad t \geq 0.
\]
Also, let \( Y_{1} \) and \( Y_{2} \) be two independent and identically distributed (i.i.d.) random variables with an exponential distribution and survival function
\[
\bar{F}_{Y}(t) = e^{-2t}, \quad t \geq 0.
\]

For these variables, we have
\[
\sigma_{X_{i}}^{2}(t) = \frac{\left(2t + 2\right) \ln\left(t + 1\right) + t^{2}}{t + 1}
\]
and
\[
L_{X_{i}}^{\sigma^{2}}(t) = \frac{t \left(t^{2} + 4t + 2\right)}{\left(t + 1\right) \left(\left(2t + 2\right) \ln\left(t + 1\right) + t^{2}\right)} < 1
\]
for all \( t > 0 \) and \( i = 1, 2 \). Additionally,
\[
L_{Y_{i}}^{\mu}(t) = 1
\]
for all \( t > 0 \) and \( i = 1, 2 \). 

Thus, \( X_{i} \leq_{VRLAI} Y_{j} \) for all \( i, j = 1, 2 \). If we define \( X = \max_{1 \leq i \leq 2} X_{i} \) and \( Y = \max_{1 \leq j \leq 2} Y_{j} \), then we have
\[
\mu_{X}(t) = \frac{\int_{t}^{\infty} \left(1 - [1 - \bar{F}_{X}(x)]^{2}\right) \, dx}{1 - [1 - \bar{F}_{X}(t)]^{2}} = \frac{2 \left(4 \left(t + 2\right) \mathrm{e}^{t} - t \left(t + 3\right)\right) - 5}{4 \left(t + 1\right) \left(2 \mathrm{e}^{t} - t - 1\right)}
\]
and
\[
\sigma_{X}^{2}(t) = \frac{2 \left(8 \left(t + 3\right) \mathrm{e}^{t} - t \left(t + 4\right)\right) - 9}{4 \left(t + 1\right) \left(2 \mathrm{e}^{t} - t - 1\right)} - \frac{\left(2 \left(4 \left(t + 2\right) \mathrm{e}^{t} - t \left(t + 3\right)\right) - 5\right)^{2}}{16 \left(t + 1\right)^{2} \left(2 \mathrm{e}^{t} - t - 1\right)^{2}}
\]

If we calculate
\[
L_{X}^{\sigma^{2}}(t) = \frac{t \cdot \sigma^{2}_{X}(t)}{\int_{0}^{t} \sigma^{2}_{X}(u) \, du},
\]
then \(0.78 < L_{X}^{\sigma^{2}}(t) < 1\). Similarly,
\[
\mu_{Y}(t) = \frac{\int_{t}^{\infty} \left(1 - [1 - \bar{F}_{Y}(y)]^{2}\right) \, dy}{1 - [1 - \bar{F}_{Y}(t)]^{2}} = \frac{4 \mathrm{e}^{2t} - 1}{8 \mathrm{e}^{2t} - 4}
\]
and
\[
\sigma_{Y}^{2}(t) = \frac{16 \mathrm{e}^{4t} - 12 \mathrm{e}^{2t} + 1}{16 \left(2 \mathrm{e}^{2t} - 1\right)^{2}}.
\]

Hence,
\[
L_{Y}^{\sigma^{2}}(t) = \frac{2t \cdot \left(16 \mathrm{e}^{4t} - 12 \mathrm{e}^{2t} + 1\right)}{\left(2 \mathrm{e}^{2t} - 1\right) \left(\left(6 \mathrm{e}^{2t} - 3\right) \ln\left(2 \mathrm{e}^{2t} - 1\right) + \left(4t - 2\right) \mathrm{e}^{2t} - 2t + 2\right)}.
\]

Here,
\[
0.92584 < L_{Y}^{\sigma^{2}}(t) < 1
\]
for all \( t > 0 \). For \( t = 0.01 \), \( L_{X}^{\sigma^{2}}(0.01) = 1 \) and \( L_{Y}^{\sigma^{2}}(0.01) = 0.9998 \). Hence, for some \( t > 0 \),
\[
\max_{1 \leq i \leq 2} X_{i} \centernot \leq_{VRLAI} \max_{1 \leq j \leq 2} Y_{j}.
\]

This shows that the VRLAI order is not closed under the formation of a coherent system.   
\end{justify}
\end{example}

 \section{Conclusion }
\begin{justify}
Understanding the ageing behaviour of a system that has already operated for a certain period ($t > 0$) is essential for assessing its reliability and performance. This paper introduced the Variance Residual Life Ageing Intensity (VRLAI) function as a quantitative measure of ageing and explored its behaviour in this context. Additionally, we examined its monotonic properties by comparing it with both the Variance Residual Life function and the Mean Residual Life function.

We introduced two new classes of ageing metrics, IVRLAI and DVRLAI, and analysed their closure properties under various reliability operations, such as distribution mixtures, convolutions of distributions, and the formation of coherent systems. Moreover, we defined an ageing order based on the Variance Residual Life function and discussed its various properties and moreover it it is demonstrated that the VRLAI order is not preserved under the formation of a coherent system.

Future research should explore the ageing behaviour of systems using the $r$-th Mean Residual Life function, with the aim of further enhancing the understanding of system reliability and performance over time.

\end{justify}

 \bibliographystyle{acm}
\bibliography{bib}
\end{document}